\documentclass[12pt, oneside]{amsart}

\usepackage{xcolor}
\usepackage{hyperref}
\hypersetup{%
	pdfpagemode={UseOutlines}, bookmarksopen, pdfstartview={FitH},
	colorlinks, linkcolor={blue}, citecolor={blue}, urlcolor={blue}
}
\usepackage{geometry}
\usepackage{amsmath, amssymb, amsthm, mathtools,amsfonts,amssymb}
\usepackage[numbers]{natbib}

\numberwithin{equation}{section}

\newtheorem{theorem}{Theorem}[section]
\newtheorem{definition}[theorem]{Definition}

\newtheorem{corollary}[theorem]{Corollary}
\newtheorem{proposition}[theorem]{Proposition}

\newtheorem{example}[theorem]{Example}
\theoremstyle{remark}
\newtheorem{remark}[theorem]{Remark}

\newcommand{\R}{\mathbb{R}}
\newcommand{\rd}{\mathbb{R}^d}
\newcommand{\rdd}{{\mathbb{R}^{2d}}}

\newcommand{\cSrd}{\mathcal{S}(\mathbb{R}^{d})}
\newcommand{\cSprd}{\mathcal{S}'(\mathbb{R}^{d})}
\newcommand{\cSprdd}{\mathcal{S}'(\mathbb{R}^{2d})}
\newcommand{\srd}{\mathcal{S}(\mathbb{R}^d)}

\newcommand{\ltrd}{L^2(\mathbb{R}^d)}
\newcommand{\ltrdd}{L^2(\mathbb{R}^{2d})}
\newcommand{\tct}{tame canonical transformation }
\DeclareMathOperator{\conesupp}{conesupp}

\title[]{A Phase Space Criterion for Dynamical Amrein--Berthier
Uncertainty}

\author{Piero D'Ancona}
\address[Piero D'Ancona]{Dipartimento di Matematica,
Sapienza Universit\`a di Roma, Piazzale A. Moro 2, 00185 Roma,
Italy}
\email{dancona@mat.uniroma1.it}

\author{Diego Fiorletta}
\address[Diego Fiorletta]{Dipartimento di Matematica,
Sapienza Universit\`a di Roma, Piazzale A. Moro 2, 00185 Roma,
Italy}
\email{diego.fiorletta@uniroma1.it}

\author{Fabio Nicola}
\address[Fabio Nicola]{Dipartimento di Scienze Matematiche,
Politecnico di Torino, Corso Duca degli Abruzzi 24, 10129 Torino,
Italy}
\email{fabio.nicola@polito.it}

\begin{document}

\keywords{Amrein--Berthier uncertainty principle,
Schr\"{o}dinger equation, Fourier integral operators, Gabor wave
front set, canonical transformations, localized propagators}
\subjclass[2020]{35Q41, 35S30, 47G30, 42B37, 81Q20}

\begin{abstract}
  We prove a phase space criterion for dynamical Amrein--Berthier
  uncertainty principles. The abstract result says that, for a
  Fourier integral operator $A\in FIO(\chi)$ associated with a
  tame canonical transformation $\chi$, the localized operator
  $\mathbf{1}_E A\mathbf{1}_F$ is compact on $L^2(\mathbb {R}^d)$
  whenever $\chi$ satisfies a vertical non refocusing condition:
  high frequency covectors issued from a spatially localized
  region cannot return to a vertical direction over the
  observation region. In the linear symplectic case this
  condition is equivalent to the familiar nondegeneracy $\det
  B\neq0$ of the upper right block of the symplectic matrix. We
  apply this compactness theorem to Schr\"{o}dinger propagators
  for Yajima--type Hamiltonians, including quadratic electric and
  linear magnetic growth, and obtain two--time Amrein--Berthier
  inequalities for compact localization sets at all nonrefocusing
  times. The result extends the compactness mechanism behind the
  dynamical Amrein--Berthier principle to a genuinely microlocal
  setting.
\end{abstract}

\maketitle

\section{Introduction}

\emph{Dynamical uncertainty principles} ask how strongly a solution 
of an evolution equation can be localized at two different times.
For Schr\"{o}dinger evolutions this question has been developed
in several directions.  Hardy and Morgan type principles form one
large strand; see, for instance, the survey
\cite{Fernandez-BertolinMalinnikova21-a}.  A second
strand concerns metaplectic and
time-frequency uncertainty principles
\cite{GroechenigShafkulovska24-a}, Hardy uncertainty for
quadratic Hamiltonians
\cite{CorderoGiacchiMalinnikova24-a}, and recent dynamical
restriction phenomena \cite{nicola2025}.  In all these problems
the propagator replaces the Fourier transform, and the geometry
of the Hamiltonian flow begins to play the role usually played by
frequency space.

The classical theorem of Amrein and Berthier
\cite{AmreinBerthier77-a}, together with Benedicks' qualitative
theorem \cite{Benedicks85-a} and the quantitative refinements of
Nazarov and Jaming \cite{Nazarov93-a,Jaming07-a}, says that a
nonzero $L^2$ function and its Fourier transform cannot both be
localized on sets of finite measure.  This is different from the
familiar Heisenberg or Hardy principles, in that the obstruction
is not the size of moments, Gaussian decay, or exponential
weights, but the non concentration of all the mass in two
``small'' regions, one in physical and one in frequency space.

In the corresponding dynamical problem the Fourier transform is
replaced by a Schr\"{o}dinger propagator $e^{-iTH}$. For the
free particle this replacement is essentially cosmetic, indeed,
at nonzero times the propagator is essentially
a rescaled Fourier transform,
up to elementary unitary factors. For a Hamiltonian
with nontrivial dynamics this identity is lost, and the
geometry of the Hamiltonian flow enters the question. The purpose
of this paper is to isolate a transparent geometric condition on
the flow which implies the compactness needed for a dynamical
Amrein--Berthier inequality.

The mechanism comes from \cite{DanconaFiorletta2026}. Given a
suitable selfadjoint $H$ and two sets $E,F$, if the localized
propagator $L=\mathbf{1}_F e^{-iTH}\mathbf{1}_E$ is compact and 
the equality $\lVert L \phi \rVert_{L^2}= \lVert \phi \rVert_{L^2}$ is ruled out
(e.g.~by unique continuation),
then one obtains a two--time estimate of the form
\begin{equation*}
  \|u(t)\|_{L^2}
  \leq
  C\left(
  \|u(0)\|_{L^2(E^c)}
  +
  \|u(T)\|_{L^2(F^c)}
  \right).
\end{equation*}
Thus the problem separates into two parts: a compactness theorem
for localized dynamics and a unique continuation argument
excluding compact support at two different times. The second part 
is classical in spirit and, in our setting, follows from local
unique continuation for magnetic Schr\"{o}dinger operators. The
novelty here is the first part. We prove compactness by using the
Gabor wave front set and the global phase space description of
Fourier integral operators developed by Tataru and by
Cordero--Gr\"{o}chenig--Nicola--Rodino
\cite{tataru2004,CGNR2013,CGNR2014}.

Let $A\in FIO(\chi)$, where $\chi$ is a tame canonical
transformation of $\mathbb{R}^{2d}$ (see Section
\ref{sec:tamecan} for definitions). 
Spatial cutoffs have Gabor wave front set
contained in the vertical cone
$\{0\}\times(\mathbb{R}^d\setminus\{0\})$. We make a basic
geometric assumption that we call
the \emph{vertical nonrefocusing condition}:
high frequency points $(0,\eta)$ in the vertical cone are not
mapped by $\chi$ back into an arbitrarily narrow vertical cone.
Under this condition, Theorem \ref{the:compact-FIO} proves that
\begin{equation*}
  \mathbf{1}_E A\mathbf{1}_F:L^2(\mathbb{R}^d)\to
  L^2(\mathbb{R}^d)
\end{equation*}
is compact for all compact sets $E,F$. The proof uses two
elementary but powerful facts: spatial localization creates only
vertical Gabor singularities, while an operator in $FIO(\chi)$
transports phase space singularities by $\chi$. Vertical
nonrefocusing prevents those singularities from returning to the
vertical cone after the propagation. The localized operator
therefore maps tempered distributions into Schwartz functions,
and compactness follows from a weighted Sobolev embedding. The
condition is also invariant under inversion of the canonical
transformation and stable under sublinear perturbations at
infinity.

The linear symplectic case gives the model test. If
\begin{equation*}
  \chi=
  \begin{pmatrix}
  A & B\\
  C & D
  \end{pmatrix}
  \in Sp(d,\mathbb{R}),
\end{equation*}
then vertical nonrefocusing is equivalent to $\det B\neq0$. This
recovers the familiar nondegeneracy behind the Fourier transform,
the free Schr\"{o}dinger propagator, and nonexceptional times for
quadratic Hamiltonians. In this sense, the condition can be
regarded as the phase space form of
the same nondegeneracy that already governs the elementary
examples.

After proving the abstract compactness theorem, we apply it to
Schr\"{o}dinger propagators associated with Yajima--type
Hamiltonians \cite{yajima1991} (see Section \ref{sec:compactness}).
These include magnetic potentials
with linear growth and electric potentials with quadratic growth,
a borderline regime not covered by the strictly
sublinear/subquadratic assumptions in
\cite{DanconaFiorletta2026}. The propagator belongs to the class
$FIO (\chi_{t,s})$, where $\chi_{t,s}$ is the associated
Hamiltonian flow,
and therefore compactness holds at all times
$(t,s)$ for which the Hamiltonian flow satisfies vertical
nonrefocusing. For quadratic Hamiltonians this reduces to the
condition $\det B_{t-s}\neq0$; for perturbations whose force is
sublinear at infinity, the stability of the condition gives the
same set of nonrefocusing times.

This gives a useful way to view the borderline examples. Linear
magnetic growth and quadratic electric growth are not small from
the point of view of the weighted Sobolev argument of
\cite{DanconaFiorletta2026}, but they are natural for the global
FIO calculus: their flows are tame canonical transformations. In
the quadratic case the exceptional times are exactly the
refocusing times of the linear symplectic flow; outside that
discrete set the dynamical Amrein--Berthier inequality follows
from the same compactness mechanism.

The final step is the dynamical Amrein--Berthier inequality,
Theorem \ref{the:dynAB}. Once compactness is known, the proof
follows the strategy of \cite{DanconaFiorletta2026}: if the norm
of the relevant localized operator were equal to 1, 
compactness would
produce a finite dimensional space of functions localized at two
times. Finite propagation for the associated wave equation,
Burnside's theorem, and unique continuation for magnetic
Schr\"{o}dinger operators then force such a function to vanish.
This contradiction gives the strict norm inequality and hence the
uncertainty estimate.  In this sense, our contribution is to
isolate, in the Amrein--Berthier setting, the geometric
compactness mechanism behind a large class of examples.  The
relevant condition is not the explicit form of the propagator,
but the absence of vertical refocusing for the associated
canonical transformation.  This gives a uniform route to
localized compactness and extends the dynamical
Amrein--Berthier inequality of \cite{DanconaFiorletta2026} to a
regime naturally governed by Fourier integral operator calculus.


The paper is organized as follows. Section \ref{sec:phase-space}
recalls the phase space tools, Gabor wave front set, and global
Fourier integral operator classes used in the proof. Section
\ref{sec:compact-FIO} proves the abstract localized compactness
theorem under vertical nonrefocusing. Section
\ref{sec:compactness} applies this result to Schr\"{o}dinger
propagators of Yajima--type and records the quadratic and
perturbed quadratic examples. Section \ref{sec:main-results}
proves the dynamical Amrein--Berthier inequality and the related
observability consequence.

\subsection*{Acknowledgments}

F.~N.~ is a Fellow of the {\em Accademia delle Scienze di Torino}
and a member of the {\em Societ\`a Italiana di Scienze e
Tecnologie Quantistiche (SISTEQ)}.

P.~D'A and D.~F.~are partially supported 
by the Progetto Ricerca Scientifica 2024
``Wave dynamics in heterogeneous media'' of Sapienza University,
and by the Gruppo Nazionale per l'Analisi Matematica, 
la Probabilit\`{a} e le loro Applicazioni (GNAMPA). D.F. is also partially supported by the Progetto Avvio alla
Ricerca 2025 ``Indeterminazione dinamica e sue applicazioni'' of Sapienza University. %

\section{Phase space tools}\label{sec:phase-space}

\subsection{Notation}

We denote by $\cSrd$ the space of Schwartz functions in $\rd$ and
by $\cSprd$ that of temperate distributions. We write $\langle
\cdot,\cdot\rangle$ for the inner product (linear in the first
argument) in $L^2(\rd)$ and $\|\cdot\|$ for the corresponding
norm. The Fourier transform of a function $f$ is normalized as
\begin{equation*}
  \mathcal{F}f(\xi)=\widehat{f}(\xi)=
  \int_{\rd} e^{-ix\cdot\xi} f(y)\, dy.
\end{equation*}
If $\Gamma$ and $\Gamma'$ are two open conic subsets of
$\rdd\setminus\{0\}$, we write $\Gamma\Subset \Gamma'$ if
$\overline{\Gamma}\cap\mathbb{S}^{2d-1}$ is a compact subset of
$\Gamma'\cap\mathbb{S}^{2d-1}$, where $\mathbb{S}^{2d-1}$ is the
unit sphere in $\rdd$ centered at the origin.

\subsection{The short-time Fourier
transform}(\cite{grochenig2001})

For $x,\xi\in\rd$ the so--called phase space shift $\pi(x,\xi)$
is the unitary operator in $\ltrd$ given by
\begin{equation*}
  \pi(x,\xi) g(y)=e^{iy\cdot\xi} g(y-x),\qquad g\in L^2(\rd).
\end{equation*}
We next define a type of localized Fourier transform. As we will
see, it provides a tool to measure the regularity and decay of a
temperate distribution.

\begin{definition}
  For $g\in\srd\setminus\{0\}$, $f\in\cSprd$ the \emph{short time
  Fourier transform (STFT)} of $f$ with \emph{window} $g$ is the
  function
  \begin{equation*}
    V_g f(z):=\langle f,\pi(z)g\rangle\qquad z\in\rdd.
  \end{equation*}
\end{definition}

One sees that if $g\in \ltrd$ and $\|g\|=(2\pi)^{-d/2}$ then
$V_g:\ltrd\to\ltrdd$ is an isometry, that is
\begin{equation*}
  V_g^\ast V_g=I.
\end{equation*}
Explicitly, for $F$ in a dense class, say
$F\in\mathcal{S}(\mathbb{R}^{2d})$, the adjoint is the synthesis
operator
\begin{equation*}
  (V_g^\ast F)(y)
  =
  \int_{\mathbb{R}^{2d}} F(x,\xi)e^{iy\cdot \xi}g(y-x)\,dx\,d\xi
  .
\end{equation*}
This inversion formula extends to distributions in $\cSprd$.

In the following, we will also need the following result (see,
e.g., \cite[Theorem 11.2.3]{grochenig2001}).

\begin{proposition}\label{pro:growth}
  If $g\in\srd\setminus\{0\}$ and $f\in\cSprd$, the function $V_g
  f$ is continuous, with a growth at most polynomial, that is,
  there exist $C,N>0$ such that
  \begin{equation*}
    |V_g f(z)|\leq C (1+|z|)^N\quad \forall z\in\rdd.
  \end{equation*}
\end{proposition}

\subsection{Global wave-front
set}(\cite{hormander1991,rodino2014})

We recall some facts concerning the so--called global (alias
Gabor, or homogeneous) wave front set of a temperate distribution
$f$, which, roughly speaking, encodes the lack of regularity or
decay of $f$.

\begin{definition}(\cite{hormander1991,rodino2014})
  With every temperate distribution $f\in \cSprd$ is associated a
  subset $WF_G(f)\subset\rdd\setminus\{0\}$ defined as the
  complement in $\rdd\setminus\{0\}$ of the set of points $z_0\in
  \rdd\setminus\{0\}$ such that there exists an open conic
  neighborhood $\Gamma\subset\rdd\setminus\{0\}$ of $z_0$ such
  that $V_g f$ has rapid decay in $\Gamma$, that is, for every
  $N\geq0$ there exists $C_N>0$ such that
  \begin{equation*}
    |V_g f(z)|\leq C_N (1+|z|)^{-N}\quad \forall z\in\Gamma.
  \end{equation*}
\end{definition}

We can show that this definition does not depend on the choice of
the window $g\in\cSrd\setminus\{0\}$. Moreover $WF_G(f)$ is a
closed conic subset of $\rdd\setminus\{0\}$.

One can see that, if $f \in \mathcal{S}'(\mathbb{R}^d)$, then
\begin{equation}\label{eq:WFS}
  f \in \mathcal{S}(\mathbb{R}^d) \iff WF_G(f) = \emptyset.
\end{equation}

\subsection{Symbol classes and microlocality}(\cite{rodino2014})

We use the notation $ S^0_{0,0}(\mathbb{R}^{2d})$ for the class
of smooth functions $a\in C^\infty(\rdd)$ such that
\begin{equation*}
  \partial^\alpha a \in L^\infty(\rdd) \quad
  \forall\alpha\in\mathbb{N}^d.
\end{equation*}

\begin{definition}
  For $a\in S^0_{0,0}(\rdd)$, the corresponding conic support
  $\conesupp a$ is the complement in
  $\mathbb{R}^{2d}\setminus\{0\}$ of the set of points $z_0 \in
  \mathbb{R}^{2d}\setminus\{0\}$ such that there exists an open
  conic neighborhood $\Gamma \subset
  \mathbb{R}^{2d}\setminus\{0\}$ of $z_0$ satisfying
  \begin{equation*}
    |a(z)| \le C_N (1+|z|)^{-N} \qquad \forall z \in \Gamma.
  \end{equation*}
\end{definition}

\begin{example}\label{exa:ex1}
  As an elementary ---yet important--- example, we observe that
  if $\varphi \in \mathcal{S}(\mathbb{R}^d)$, regarded as a
  function in $S^0_{0,0}(\mathbb{R}^{2d})$ independent of $\xi$,
  we obtain
  \begin{equation*}
    \conesupp \varphi = \{0\} \times (\mathbb{R}^d \setminus
    \{0\}).
  \end{equation*}
\end{example}

For a given symbol $a\in\cSprdd$, we consider the corresponding
Weyl quantization $a^w$, as a linear continuous operator
$\cSrd\to\cSprd$ defined formally as
\begin{equation*} a^w f(x)=(2\pi)^{-d}\int_\rdd e^{i
  (x-y)\cdot\xi}a\Big(\frac{x+y}{2},\xi\Big)f(y)\, dy\, d\xi.
\end{equation*}
If $a\in S^0_{0,0}$ the operator $a^w$ extends to a linear
continuous map in $\cSrd$, $\cSprd$, and $\ltrd$.

We also recall that the operators $a^w$, with $a \in
S^0_{0,0}(\mathbb{R}^{2d})$, are globally microlocal, in the
following sense.

\begin{proposition}\label{pro:micolocality}
  We have
\begin{equation*}
  a \in S^0_{0,0}(\mathbb{R}^{2d}),\ f \in
  \mathcal{S}'(\mathbb{R}^d)
  \;\Rightarrow\;
  WF_G(a^w f) \subset \conesupp a \cap WF_G(f).
  \end{equation*}
\end{proposition}

This result was first established in this generality in
\cite{rodino2014} (see also \cite{hormander1991} and the recent
contribution \cite[Section 7]{nicola2025}, where it is presented
exactly in the form recalled here).

\subsection{Tame canonical transformations and associated
operators} (\cite{CGNR2013,tataru2004}). \label{sec:tamecan}

We now define a class of canonical transformations which arise
naturally when considering quadratic type Hamiltonians (see
\cite{yajima1991} and Section \ref{sec:compactness} below).

\begin{definition}
  A \emph{\tct}is a smooth map $\chi:\rdd\to\rdd$ such that
  \begin{itemize}
    \item $|\partial^\alpha \chi(z)|\leq C_\alpha$ for all
    $|\alpha|\geq 1$, $z\in\rdd$;

    \item $\chi$ is a canonical transformation, that is, if
    $(x,\xi)=\chi(y,\eta)$ we have $dy\wedge d\eta=dx\wedge
    d\xi$.
  \end{itemize}
\end{definition}

It is easy to see that every \tct is global bi--Lipschitz. Also,
the class of tame canonical transformations is closed with
respect to inversion and composition.

\begin{definition}(\cite{CGNR2013,tataru2004})
  Given a tame canonical transformation
  $\chi:\mathbb{R}^{2d}\to\mathbb{R}^{2d}$, we consider the class
  $FIO(\chi)$ of operators
  $A:\mathcal{S}(\mathbb{R}^d)\to\mathcal{S}'(\mathbb{R}^d)$ such
  that, for some (and hence any)
  $g\in\mathcal{S}(\mathbb{R}^d)\setminus\{0\}$ and for every
  $N>0$ there exists $C_N>0$ satisfying
  \begin{equation}\label{eq:1star}
    \bigl|\langle A\,\pi(z)g,\, \pi(w)g\rangle\bigr|
    \le C_N \,(1+|w-\chi(z)|)^{-N}
    \qquad \forall z,w\in\mathbb{R}^{2d}.
  \end{equation}
\end{definition}

In fact, in \cite{tataru2004} the particular case of a Gaussian
window $g$ was considered, while the extension to a general
window $g\in\cSrd\setminus\{0\}$ was considered in
\cite{CGNR2013}. It was also observed there, see \cite[Lemma
3.3]{CGNR2013}, that this class of operators is independent of
the choice of $g$. It is easy to see that any operator $A\in
FIO(\chi)$ extends to a linear continuous operator in $\cSrd$,
$\ltrd$ and $\cSprd$ (see \cite{CGNR2013}).

\begin{remark}
  We observe that when $\chi$ is linear, the operators in
  $FIO(\chi)$ are the so--called \textit{generalized metaplectic
  operators}. They can be characterized as operators of the form
  $A=S\,b^w$, where $S$ is a metaplectic operator and $b\in
  S^0_{0,0}(\mathbb{R}^{2d})$; see \cite{CGNR2014}.
\end{remark}

The following elementary but important property was observed in
\cite[Proposition 6.5]{tataru2004} and \cite[Theorem
3.6]{CGNR2013}, and will be used in the following.

\begin{proposition}\label{pro:composition}
  If $\chi_1$ and $\chi_2$ are tame canonical transformations,
  $A_1\in FIO(\chi_1)$ and $A_2\in FIO(\chi_2)$, then $A_1 A_2\in
  FIO(\chi_1\circ \chi_2)$.
\end{proposition}

\subsection{Generating functions and associated operators}
\label{sec:ganerating-functions}
(\cite{arnold,stein1993,tataru2004})

It is well known from classical mechanics (see, e.g.,
\cite{arnold}) that canonical transformations can be manipulated
in an economic way by means of generating functions. We now
discuss this connection in the special case of tame canonical
transformations.

Following \cite{tataru2004}, for $k\in \mathbb{N}$ we define the
class $ S^{(k)}_{0,0}(\mathbb{R}^{2d}) $ of functions $a\in
C^\infty(\mathbb{R}^{2d})$ such that
\begin{equation*}
  \partial^\alpha a \in L^\infty
  \qquad \text{for }|\alpha|\ge k.
\end{equation*}

\begin{definition}
  If $\phi\in S^{(2)}_{0,0}(\mathbb{R}^{2d})$ and
  \begin{equation*}
    \bigl|\det\, \partial^2_{x\xi}\phi(x,\xi)\bigr|
    \ge \delta>0\quad\forall(x,\xi)\in\rd\times\rd
  \end{equation*}
  for some $\delta>0$, we say that $\phi$ is a \emph{tame phase}.
\end{definition}

Let $\chi:\mathbb{R}^{2d}\to\mathbb{R}^{2d}$ be a tame canonical
transformation. We write
\begin{equation*}
  (x,\xi)=\chi(y,\eta),
\end{equation*}
or equivalently
\begin{equation*}
  x=x(y,\eta),\qquad \xi=\xi(y,\eta).
\end{equation*}
Then $dx\wedge d\xi=dy\wedge d\eta$, hence
\begin{equation*}
  \xi\,dx-\eta\,dy=d\tilde S
\end{equation*}
for some function $\tilde S(y,\eta)$ in $S^{(2)}_{0,0}(\rdd)$,
uniquely determined up to an additive constant. Suppose that, for
some $\delta>0$,
\begin{equation*}
  \bigl|\det \tfrac{\partial x}{\partial \eta}(y,\eta) \bigr|
  \ge \delta>0
  \qquad \forall (y,\eta)\in\mathbb{R}^{d}\times \mathbb{R}^{d}.
\end{equation*}
Then, by Hadamard's global inverse function theorem (see, e.g.,
\cite[Theorem 6.2.4]{krantz}), we can globally solve the equation
\begin{equation*}
  x=x(y,\eta)
\end{equation*}
with respect to $\eta$ and obtain a function $\eta=\eta(x,y)$ in
$S^{(1)}_{0,0}(\mathbb{R}^{2d})$. Substituting in
$\tilde{S}(y,\eta)$ we obtain a function $S(x,y)$, called the
\emph{generating function} of $\chi$, which satisfies
\begin{equation*}
  \xi\,dx-\eta\,dy=dS.
\end{equation*}
Hence
\begin{equation*}
  \frac{\partial S}{\partial y}(x,y)=-\eta(x,y),
  \qquad
  \frac{\partial S}{\partial x}(x,y)=\xi\bigl(y,\eta(x,y)\bigr).
\end{equation*}
Moreover $S(x,y)$ is a tame phase.

Conversely, given a tame phase $S(x,y)$, the equations
\begin{equation*}
  \frac{\partial S}{\partial y}(x,y)=-\eta,
  \qquad
  \frac{\partial S}{\partial x}(x,y)=\xi
\end{equation*}
define a unique \tct $(x,\xi)=\chi(y,\eta)$.

Similarly, if $(x,\xi)=\chi(y,\eta)$ is a tame canonical
transformation, we have
\begin{equation*}
  \xi\,dx+y\,d\eta=d\bigl(\tilde S(y,\eta)+y\eta\bigr).
\end{equation*}
Suppose that, for some $\delta>0$,
\begin{equation*}
  \bigl|\det \tfrac{\partial x}{\partial y}(x,\eta)\bigr|
  \ge \delta>0\quad\forall(x,\eta)\in\rd\times\rd.
\end{equation*}
Then we can globally solve the equation $x=x(y,\eta)$ with
respect to $y$ and find a function $y=y(x,\eta)$ in
$S^{(1)}_{0,0}(\mathbb{R}^{2d})$. Substituting into $\tilde
S(y,\eta)+y\eta$ we obtain a function $\Phi(x,\eta)$ in
$S^{(2)}_{0,0}(\rdd)$, and we have
\begin{equation*}
  \xi\,dx+y\,d\eta=d\Phi.
\end{equation*}
Hence
\begin{equation*}
  \frac{\partial \Phi}{\partial x}(x,\eta)
  =\xi\bigl(y(x,\eta),\eta\bigr),
  \qquad
  \frac{\partial \Phi}{\partial \eta}(x,\eta)=y(x,\eta).
\end{equation*}
Moreover, $\Phi$ is a tame phase, called the generating function
of $\chi$ in the coordinates $(x,\eta)$.

Conversely, given a tame phase $\Phi(x,\eta)$, the equations
\begin{equation*}
  \frac{\partial \Phi}{\partial x}(x,\eta)=\xi,
  \qquad
  \frac{\partial \Phi}{\partial \eta}(x,\eta)=y
\end{equation*}
define a unique tame canonical transformation $\chi$.

With these two types of generating functions, $S(x,y)$ and
$\Phi(x,\eta)$, we associate two classes of integral operators:

\begin{itemize}
  \item \emph{Oscillatory integral operators} (OIO):
  \begin{equation*}
    Af(x)=\int_{\mathbb{R}^d} e^{iS(x,y)}\,a(x,y)\,f(y)\,dy,
  \end{equation*}
  where $S(x,y)$ is a tame phase and $a\in
  S^{0}_{0,0}(\mathbb{R}^{2d})$.

  \item \emph{Classical Fourier integral operators} (FIO):
  \begin{equation*}
    Af(x)=
    \int_{\mathbb{R}^d} e^{i\Phi(x,\eta)}\,b(x,\eta)\,
    \widehat f(\eta)\,d\eta,
  \end{equation*}
  where $\Phi(x,\eta)$ is a tame phase and $b\in
  S^{0}_{0,0}(\mathbb{R}^{2d})$.
\end{itemize}

One can see that they are linear continuous operators in $\cSrd$,
$\ltrd$ and $\cSprd$ (see, e.g., \cite{stein1993}).

We recall the following fact from \cite[Theorem 4.3]{CGNR2013}.

\begin{proposition}\label{pro:fio}
  A classical FIO $A$, with a tame phase $\Phi(x,\eta)$, belongs
  to the class $FIO(\chi)$, where $\chi$ is the tame canonical
  transformation associated with $\Phi$.
\end{proposition}

We now verify that a similar property holds for an OIO with a
tame phase $S(x,y)$.

\begin{proposition}\label{pro:CT}
  An OIO $A$, with a tame phase $S(x,y)$, belongs to the class
  $FIO(\chi)$, where $\chi$ is the tame canonical transformation
  associated with $S$.
\end{proposition}

\begin{proof}
An OIO $A$ with tame phase $S(x,y)$ and amplitude $a(x,y)$ can be
written as $A=A'\circ\mathcal{F}^{-1}$, where $A'$ is a classical
FIO with phase $S(x,\eta)$ and amplitude $a(x,\eta)$. Hence, by
Proposition \ref{pro:fio}, $A'\in FIO(\chi')$ for a tame
canonical transformation
\begin{equation*}
  (x,\xi)=\chi'(y,\eta)
\end{equation*}
implicitly defined by
\begin{equation*}
  \frac{\partial S}{\partial x}(x,\eta)=\xi,
  \qquad
  \frac{\partial S}{\partial \eta}(x,\eta)=y.
\end{equation*}
Hence
\begin{equation*}
  \bigl|\langle A'\,\pi(y,\eta)g,\, \pi(x,\xi)g\rangle\bigr|
  \le C_N \,(1+|(x,\xi)-\chi'(y,\eta)|)^{-N}.
\end{equation*}
Replacing in this formula the pair $(y,\eta)$ with the pair
$(-\eta,y)$, choosing $g$ such that $\mathcal{F}^{-1}g=g$ (for
example, a suitable Gaussian) and using that
\begin{equation*}
  \mathcal{F}^{-1} \pi(y,\eta)g
  =c\,  \pi(-\eta,y)\mathcal{F}^{-1} g,
\end{equation*}
with $|c|=1$, we see that $A\in FIO(\chi)$, with $\chi=\chi'\circ
J^{-1}$, where
\begin{equation}\label{eq:J}
  J=\begin{pmatrix}0 & I \\ -I & 0\end{pmatrix},
  \qquad
  J^{-1}=\begin{pmatrix}0 & -I \\ I & 0\end{pmatrix}.
\end{equation}
Therefore, if we introduce the variables $(y', \eta')$ via
\begin{equation*}
  \begin{pmatrix} y \\ \eta \end{pmatrix}
  = J^{-1}\begin{pmatrix} y' \\ \eta' \end{pmatrix}
  = \begin{pmatrix} -\eta' \\ y' \end{pmatrix}
\end{equation*}
we see that
\begin{equation*}
  (x,\xi)=\chi'(y,\eta)=\chi'\circ J^{-1}(y',\eta')
  =\chi(y',\eta')
\end{equation*}
and
\begin{equation*}
  \frac{\partial S}{\partial x}(x,y')=\xi,
  \qquad
  \frac{\partial S}{\partial y'}(x,y')=-\eta'.
\end{equation*}
Hence $\chi$ is precisely the tame canonical transformation
associated with the generating function $S$.
\end{proof}

\section{Compactness of localized Fourier integral operators}
\label{sec:compact-FIO}

Consider a \tct $\chi:\rdd\to\rdd$. We now make a fundamental
assumption:
\begin{equation}\label{eq:2star}
  \begin{aligned}
    &\text{There exists an open conic neighborhood }
    \Gamma\subset \rdd\setminus\{0\}
    \text{ of }\{0\}\times(\mathbb{R}^d\setminus\{0\}) \\
    &\text{and $R>0$ such that } \eta\in\rd,\ |\eta|>R
    \;\Rightarrow\; \chi(0,\eta)\not \in\Gamma.
  \end{aligned}
\end{equation}
We shall refer to \eqref{eq:2star} as the \emph{vertical
nonrefocusing condition}. It says that high-frequency vertical
covectors based over a spatially localized region cannot be
mapped by $\chi$ back into a vertical direction.

\begin{remark}[The linear symplectic test case]\label{rem:linear}
  Suppose that $\chi$ is linear, hence represented by a matrix of
  the form
  \begin{equation*}
    \begin{pmatrix}
    A & B\\
    C & D
    \end{pmatrix}
    \in Sp(d,\mathbb{R}).
  \end{equation*}
  Then condition \eqref{eq:2star} is equivalent to
  \begin{equation*}
    |B \eta| \geq \varepsilon |D \eta|
    \qquad \forall\eta\in\mathbb{R}^d,
  \end{equation*}
  for some $\varepsilon>0$. In turn, this is equivalent to $\det
  B\neq 0$, since $\chi$ is a symplectic matrix, hence
  nondegenerate.
\end{remark}

\begin{proposition}\label{pro:symm}
  The vertical nonrefocusing condition \eqref{eq:2star} is
  satisfied by a tame canonical transformation $\chi$ if and only
  if it is satisfied by $\chi^{-1}$.
\end{proposition}
\begin{proof}
It is best to work with the logical negation of \eqref{eq:2star}.

We observe that $\chi$ does not satisfy \eqref{eq:2star} if and
only if there exist sequences $x_n,\xi_n,\eta_n\in\rd$, with
$\xi_n,\eta_n\not=0$, such that
\begin{equation*}
    (x_n,\xi_n)=\chi(0,\eta_n),
\end{equation*}
\begin{equation*}
    |\eta_n|\to+\infty\qquad |x_n|/|\xi_n|\to0.
\end{equation*}
Assuming this, we observe that, since $\chi^{-1}$ is globally
Lipschitz and $|x_n|\leq C|\xi_n|$, we have $|\xi_n|\to+\infty$.
Moreover, if we define
\begin{equation*}
  (y'_n,\eta'_n):=\chi^{-1}(0,\xi_n),
\end{equation*}
we have
\begin{align*}
  |y'_n|=|y'_n-0|
  &\leq |\chi^{-1}(0,\xi_n)-\chi^{-1}(x_n,\xi_n)|\\
  &\leq C|x_n|\\
  &=C\varepsilon_n|\xi_n|,
\end{align*}
with $\varepsilon_n= |x_n|/|\xi_n|\to0$, where we used again that
$\chi^{-1}$ is globally Lipschitz. Since
$(0,\xi_n)=\chi(y'_n,\eta'_n)$ and $\chi$ is globally Lipschitz,
we obtain
\begin{equation*}
  |y'_n|+|\eta'_n|\to+\infty
\end{equation*}
and
\begin{equation*}
    |y'_n|\leq C' \varepsilon_n(1+|y'_n|+|\eta'_n|),
\end{equation*}
which implies $|y'_n|/|\eta'_n|\to 0$.

This means that $\chi^{-1}$ does not satisfy \eqref{eq:2star}.
\end{proof}
The following result deals with the robustness of vertical
nonrefocusing.

\begin{proposition}\label{pro:pertur}
Let $\chi$ and $\chi'$ be two canonical transformations such that
\begin{equation*}
  \lim_{|\eta|\to+\infty}
  \frac{\bigl|\chi(0,\eta)-\chi'(0,\eta)\bigr|}{|\eta|}=0.
\end{equation*}
Then, if $\chi$ satisfies \eqref{eq:2star}, also $\chi'$
satisfies \eqref{eq:2star}.
\end{proposition}

\begin{proof}
Let $(x,\xi)=\chi(0,\eta)$ and $(x',\xi')=\chi'(0,\eta)$, for
$\eta\in\mathbb{R}^d$. By assumption, there exist $R>0$ and
$\varepsilon>0$ such that
\begin{equation*}
  |\eta|>R \ \Rightarrow\ |x|\ge \varepsilon|\xi|.
\end{equation*}
Hence, for $|\eta|>R'\ge R$,
\begin{equation*}
  |x'|\ge \varepsilon|\xi'|-\lvert x-x'\rvert
  -\varepsilon\lvert \xi-\xi'\rvert
  \ge \varepsilon|\xi'|-c_{R'}|\eta|,
\end{equation*}
where $c_{R'}\to 0$ as $R'\to+\infty$. Since $(\chi')^{-1}$ is
globally Lipschitz, we have
\begin{equation*}
  |\eta|\le C\,(1+|x'|+|\xi'|),
\end{equation*}
and therefore
\begin{equation*}
  |x'|\ge \frac{\varepsilon}{2}\bigl(|\xi'|-1\bigr)
  \quad\text{and}\quad
  2\le |x'|+|\xi'|
\end{equation*}
for $R'$ sufficiently large, which implies
\begin{equation*}
  |x'|\ge \frac{\varepsilon}{4+\varepsilon}\,|\xi'|.
\end{equation*}
\end{proof}
We can now state the main result of this section.
\begin{theorem}[Localized compactness under vertical
nonrefocusing]\label{the:compact-FIO} Let $\chi$ be a tame
canonical transformation satisfying the vertical nonrefocusing
condition \eqref{eq:2star}. Let $A\in FIO(\chi)$. Then, if
$E,F\subset\mathbb{R}^d$ are compact subsets, the operator
$\mathbf{1}_E\,A\,\mathbf{1}_F$ is compact in
$L^2(\mathbb{R}^d)$.
\end{theorem}

\begin{proof}
Let $\varphi\in\mathcal{S}(\mathbb{R}^d)$, with $\varphi=1$ on
$E\cup F$. Then
\begin{equation*}
  \mathbf{1}_E\,A\,\mathbf{1}_F
  =
  \mathbf{1}_E\,\varphi\,A\,\varphi\,\mathbf{1}_F.
\end{equation*}
Therefore, it is sufficient to show that $\varphi A \varphi$ is
compact in $L^2(\rd)$.

Now, $\varphi A \varphi$ extends to a continuous map
$\mathcal{S}'(\mathbb{R}^d)\to\mathcal{S}'(\mathbb{R}^d)$, and we
will prove that, in fact, it maps $\mathcal{S}'(\mathbb{R}^d)$
into $\mathcal{S}(\mathbb{R}^d)$. Assuming this latter fact, if
we consider the weighted Sobolev space
\begin{equation*}
  X:=(1+x^2)(1-\Delta) \big(L^2(\rd)\big),
\end{equation*}
we see that $\varphi A\varphi:X\to L^2(\rd)$ continuously by the
closed graph theorem, and the conclusion follows, since the
embedding $L^2(\rd) \hookrightarrow X$ is compact (see, e.g.,
\cite{nicola2010}).

Therefore, we must show that if $f\in\mathcal{S}'(\mathbb{R}^d)$,
then $\varphi A \varphi f\in\mathcal{S}(\mathbb{R}^d)$, that is,
by \eqref{eq:WFS},
\begin{equation*}
  WF_G(\varphi A \varphi f)=\emptyset.
\end{equation*}
Viewing $\varphi$ as a symbol in $S^0_{0,0}(\mathbb{R}^{2d})$
independent of $\xi$, we have, by Example \ref{exa:ex1},
\begin{equation*}
  \conesupp\varphi=\{0\}\times(\mathbb{R}^d\setminus\{0\}).
\end{equation*}
Hence, by Proposition \ref{pro:micolocality}, for all
$f\in\mathcal{S}'(\mathbb{R}^d)$ we have
\begin{equation*}
  WF_G(\varphi A \varphi f)
  \subset
  \{0\}\times(\mathbb{R}^d\setminus\{0\})
  \cap WF_G(A\varphi f).
\end{equation*}
Thus it suffices to show that this last intersection is empty. To
this end, using the inversion formula $V_g^*V_g=I$ in $\cSprd$
(assuming $\|g\|_{L^2}=(2\pi)^{-d/2}$; see Section
\ref{sec:phase-space}), we can write, in the weak sense,
\begin{equation*}
  \varphi f =
  \int_{\mathbb{R}^{2d}} V_g(\varphi f)(z)\,\pi(z)g \, dz
\end{equation*}
and therefore, for $w \in \mathbb{R}^{2d}$,
\begin{equation*}
  V_g(A\varphi f)(w) =
  \int_{\mathbb{R}^{2d}} V_g(\varphi f)(z)\,
  \langle A\pi(z)g,\pi(w)g\rangle \, dz.
\end{equation*}
Using \eqref{eq:1star}, for every $N \in \mathbb{N}$ there exists
$C_N > 0$ such that
\begin{equation*}
  \lvert V_g(A\varphi f)(w)\rvert
  \le C_N \int_{\mathbb{R}^{2d}}
  (1+\lvert w-\chi(z)\rvert)^{-N}\,
  \lvert V_g(\varphi f)(z)\rvert \, dz.
\end{equation*}
Since $\chi$ is a canonical transformation, its Jacobian is $=1$,
and therefore
\begin{equation*}
  \lvert V_g(A\varphi f)(w)\rvert
  \le C_N \int_{\mathbb{R}^{2d}} (1+\lvert w-z\rvert)^{-N}\,
  \lvert V_g(\varphi f)(\chi^{-1}(z))\rvert \, dz.
\end{equation*}
We have to prove that this function of $w$ has fast decay in a
conic neighborhood of $\{0\}\times(\mathbb{R}^d\setminus\{0\})$.

For $\varepsilon>0$, let
\begin{equation*}
  \Gamma_\varepsilon :=
  \{(x,\xi)\in\mathbb{R}^{2d}\setminus\{0\} :
  \lvert x\rvert < \varepsilon \lvert \xi\rvert\}.
\end{equation*}
By the assumption and Proposition \ref{pro:symm}, the map
$\chi^{-1}$ satisfies \eqref{eq:2star} as well, that is
\begin{equation}\label{eq:gammadelta}
  \begin{aligned}
    &\text{There exist $\delta>0$ and $R>0$ such that}\\
    &\xi\in\rd,\ |\xi|>R
    \;\Rightarrow\; \chi^{-1}(0,\xi)\not \in \Gamma_\delta.
  \end{aligned}
\end{equation}
We show that if $\varepsilon$ is sufficiently small, possibly
taking a smaller $\delta$ and a larger $R$ we have
\begin{equation}
  (x,\xi)\in \Gamma_\varepsilon,\ \lvert \xi\rvert \ge R
  \ \Longrightarrow\
  \chi^{-1}(x,\xi)\in \Gamma_\delta^{\,c}
\end{equation}
where $\Gamma_\delta^{\,c}$ stands for the complement of
$\Gamma_\delta$ in $\rdd\setminus\{0\}$.

Indeed, by \eqref{eq:gammadelta} and setting
$(y,\eta)=\chi^{-1}(x,\xi)$, we have
\begin{equation*}
  |y(0,\xi)| \ge \delta\,|\eta(0,\xi)|
  \qquad \text{for }|\xi|\ge R.
\end{equation*}
Moreover
\begin{equation*}
  |y(x,\xi)-y(0,\xi)| + |\eta(x,\xi)-\eta(0,\xi)|
  \le C\,|x|.
\end{equation*}
since $\chi^{-1}$ is a tame canonical transformation, hence
globally Lipschitz.

Therefore, if $(x,\xi)\in\Gamma_\varepsilon$ (thus
$|x|<\varepsilon|\xi|$),
\begin{align*}
  |y(x,\xi)|
  &\ge \delta\,|\eta(x,\xi)| - C'|x| \\
  &\ge \delta\,|\eta(x,\xi)| - C'\varepsilon|\xi| \\
  &\ge \delta\,|\eta(x,\xi)|
  - C''\varepsilon\bigl(1+|\eta(x,\xi)|+|y(x,\xi)|\bigr),
\end{align*}
where in the last step we used that $\chi$ is tame too, hence
globally Lipschitz.

Thus, if $\varepsilon$ is sufficiently small,
\begin{equation*}
  |y(x,\xi)| \ge C\,|\eta(x,\xi)| - C'
  \qquad \text{for }(x,\xi)\in\Gamma_\varepsilon,\ |\xi|\ge R,
\end{equation*}
for suitable constants $C,C'>0$.

If $|\xi|\ge R$, we also have, since $\chi$ is globally
Lipschitz,
\begin{equation*}
  1+|y(x,\xi)|+|\eta(x,\xi)| \ge C''|\xi| \ge C''R.
\end{equation*}
Hence, if $R$ is sufficiently large, $|y(x,\xi)| \ge
\delta\,|\eta(x,\xi)|$ for some $\delta>0$. Therefore, possibly
for a smaller $\delta$ and a larger $R$, we have
\begin{equation}\label{eq:3star}
  (x,\xi)\in\Gamma_\varepsilon,\ |\xi|\ge R
  \ \Rightarrow\ \chi^{-1}(x,\xi)\in\Gamma_\delta^{\,c}.
\end{equation}
Moreover, we may assume $\varepsilon<\delta$, so that
$\Gamma_\varepsilon\Subset\Gamma_\delta$.

We now return to the last integral, for which we now prove rapid
decay for $w\in\Gamma_{\varepsilon'}$, for every
$\varepsilon'<\varepsilon$. We observe that
\begin{equation*}
  \Gamma_{\varepsilon'}\Subset
  \Gamma_{\varepsilon}\Subset
  \Gamma_\delta.
\end{equation*}
For every $M>0$, $N>0$ we have
\begin{equation*}
  (1+|w|)^M\,|V_g(A\varphi f)(w)|
  \le C_N \int_{\mathbb{R}^{2d}}
  \frac{(1+|w|)^M}{(1+|z-w|)^N}\,
  \bigl|V_g(\varphi f)\bigl(\chi^{-1}(z)\bigr)\bigr|\,dz
\end{equation*}
and we split the last integral as
\begin{equation*}
  C_N\left[
  \int_{\Gamma^{(1)}_{\varepsilon,R}}\cdots+
  \int_{\Gamma^{(2)}_{\varepsilon,R}}\cdots+
  \int_{\Gamma_\varepsilon^{\,c}}\cdots\right],
\end{equation*}
where
\begin{equation*}
  \Gamma^{(1)}_{\varepsilon,R}:=
  \Gamma_\varepsilon\cap\{(x,\xi)\in\rdd:\ |\xi|\leq R\},
\end{equation*}
and
\begin{equation*}
  \Gamma^{(2)}_{\varepsilon,R}:=
  \Gamma_\varepsilon\cap\{(x,\xi)\in\rdd:\ |\xi|>R\}.
\end{equation*}
We estimate the three integrals separately, always for
$w\in\Gamma_{\varepsilon'}$.

\textit{Integral over $\Gamma_\varepsilon^{\,c}$}. If
$z\in\Gamma_\varepsilon^{\,c}$ and $w\in\Gamma_{\varepsilon'}$,
then
\begin{equation*}
  |z-w|\ge C\,(|z|+|w|)
\end{equation*}
for some $C>0$. Indeed, if $\theta$ denotes the angle between $z$
and $w$ in $\rdd$, we have $\cos\theta <\lambda<1$, since the
intersections of $\overline{\Gamma_{\varepsilon'}}$ and
$\Gamma_\varepsilon^{\,c}$ with the unit sphere in
$\mathbb{R}^{2d}$ are disjoint compact sets, and hence have
positive distance. Hence
\begin{equation*}
  |z-w|^2=|z|^2+|w|^2-2|z||w|\cos\theta
  \ge (1-\lambda)(|z|^2+|w|^2).
\end{equation*}

Therefore, since the function $V_g(\varphi f)$ has at most
polynomial growth by Proposition \ref{pro:growth}, the integral
over $\Gamma_\varepsilon^{\,c}$ can be estimated, for some $k>0$,
by
\begin{equation*}
  C'\int_{\mathbb{R}^{2d}}
  \frac{(1+|w|)^M}{(1+|z|)^{N/2}(1+|w|)^{N/2}}\,
  \bigl(1+|\chi^{-1}(z)|\bigr)^k\,dz
  < C'',
\end{equation*}
provided that $N/2>M$ and $N/2>k+2d$, since $\chi^{-1}$ is
globally Lipschitz.

\textit{Integral over $\Gamma^{(1)}_{\varepsilon,R}$}. The
desired estimate can be obtained similarly; indeed, in this case
$|z|$ is bounded and $1+|z-w|\ge C(1+|w|)$.

\textit{Integral over $\Gamma^{(2)}_{\varepsilon,R}$}. Here we
use \eqref{eq:3star}: if $z=(x,\xi)\in\Gamma_\varepsilon$ and
$|\xi|>R$, then $\chi^{-1}(z)\in\Gamma_\delta^{\,c}$. On the
conic set $\Gamma_\delta^{\,c}$ the function $V_g(\varphi
f)$ has fast decay, since by Proposition \ref{pro:micolocality}
and Example \ref{exa:ex1},
\begin{equation*}
  WF_G(\varphi f)\subset\{0\}\times(\mathbb{R}^d\setminus\{0\})
  \cap
  WF_G(f)\subset\{0\}\times(\mathbb{R}^d\setminus\{0\}).
\end{equation*}
Therefore, the integral over $\Gamma^{(2)}_{\varepsilon,R}$ can
be estimated, for every $N,N'>0$, as
\begin{equation*}
  C_{N,N'}\int_{\mathbb{R}^{2d}}
  \frac{(1+|w|)^M}{(1+|z-w|)^N}\,
  \frac{1}{(1+|\chi^{-1}(z)|)^{N'}}\,dz
\end{equation*}
\begin{equation*}
  \le C'_{N,N'}\int_{\mathbb{R}^{2d}}
  \frac{(1+|w|)^M}{(1+|z-w|)^N}\,
  \frac{1}{(1+|z|)^{N'}}\,dz
  \le C''_{N,N'},
\end{equation*}
provided that $N,N'$ are sufficiently large.

This concludes the proof of the theorem.
\end{proof}
\begin{remark}
  Observe that the operator $\mathbf{1}_E\,A\,\mathbf{1}_F$ in
  Theorem \ref{the:compact-FIO} is compact in $L^2(\mathbb{R}^d)$
  if and only if its adjoint $\mathbf{1}_F\,A^\ast\,\mathbf{1}_E$
  is compact. On the other hand, one easily sees that $A^\ast\in
  FIO(\chi^{-1})$. Hence, in view of Proposition \ref{pro:symm},
  the statement of Theorem \ref{the:compact-FIO} is symmetric
  under the exchange of $\chi$ with $\chi^{-1}$ and $A$ with
  $A^\ast$.
\end{remark}

\section{Compactness of the localized Schr\"{o}dinger propagator}
\label{sec:compactness}

We now apply Theorem \ref{the:compact-FIO} to the propagator for
the class of Schr\"{o}dinger equations in $\rd$ considered by
Yajima \cite{yajima1991}.

Let $s\in\R$, and consider the Cauchy problem in $\R\times\rd$:
\begin{equation*}
  \begin{cases}
    i\partial_t u = Hu,\\
    u(s,x)=f(x),
  \end{cases}
\end{equation*}
where the Hamiltonian $H$ belongs to the following class.

\begin{definition}\label{def:yajima}
  We say that $H$ is a Yajima--type Hamiltonian if it has the form
  \begin{equation*}
    H=\tfrac12\bigl(i\nabla+A(t,x)\bigr)^2+V(t,x),
    \quad t \in \mathbb{R},\ x \in \mathbb{R}^d,
  \end{equation*}
  where $V(t,x)$ and $A(t,x) = (A_1(t,x), \ldots, A_d(t,x))$ are
  the electric scalar and magnetic vector potential of the field,
  satisfying the following hypotheses.

  \begin{itemize}
    \item[(a)]
    For $j = 1, \ldots, d$, $A_j(t,x)$ is a real function of
    $(t,x) \in \mathbb{R} \times \mathbb{R}^d$ and
    $\partial_x^\alpha A_j(t,x)$ is $C^1$ in $(t,x) \in
    \mathbb{R} \times \mathbb{R}^d$, for every $\alpha \in
    \mathbb{N}^d$. Moreover there exists $\varepsilon > 0$ such
    that
    \begin{equation*}
      |\partial_x^\alpha B(t,x)|
      \le C_\alpha (1 + |x|)^{-1-\varepsilon},
      \qquad |\alpha| \ge 1,
    \end{equation*}
    \begin{equation*}
      |\partial_x^\alpha A(t,x)|
      + |\partial_x^\alpha \partial_t A(t,x)|
      \le C_\alpha,
      \qquad |\alpha| \ge 1,
    \end{equation*}
    for $(t,x) \in \mathbb{R} \times \mathbb{R}^d$, where
    $B(t,x)$ is the magnetic field, i.e.\ the skew-symmetric
    matrix with entries
    \begin{equation*}
      B_{j,k}(t,x) =
      \partial_{x_j} A_k(t,x) - \partial_{x_k} A_j(t,x).
    \end{equation*}

    \item[(b)]
    $V(t,x)$ is a real function of $(t,x) \in \mathbb{R} \times
    \mathbb{R}^d$, with $\partial^\alpha_x V(t,x)$ continuous in
    $(t,x) \in \mathbb{R} \times \mathbb{R}^d$ for every $\alpha
    \in \mathbb{N}^d$ and satisfying
    \begin{equation*}
      |\partial_x^\alpha V(t,x)| \le C_\alpha,
      \qquad |\alpha| \ge 2,
      \quad (t,x) \in \mathbb{R} \times \mathbb{R}^d.
    \end{equation*}
  \end{itemize}
\end{definition}

It was proved in \cite[Section 2]{yajima1991} that under these
assumptions the Hamiltonian
\begin{equation*}
  H(x,\xi)=\frac{1}{2}(\xi-A(t,x))^2+V(t,x)
\end{equation*}
generates a flow
\begin{equation*}
  \chi_{t,s}:\mathbb{R}^{2d}\to\mathbb{R}^{2d},
  \qquad t,s\in\mathbb{R},
\end{equation*}
which is globally defined. It is a two-parameter family of tame
canonical transformations, that is, $\partial^\alpha\chi_{t,s}\in
L^\infty$ for $|\alpha|\ge1$. Moreover,
\begin{equation*}
  \chi_{t,s}\circ\chi_{s,u}=\chi_{t,u}
  \qquad \forall\,t,s,u\in\mathbb{R}.
\end{equation*}
One of the main results of \cite{yajima1991} (see \cite[Theorem 4
and Remark (a)]{yajima1991}) states that, if $0<|t-s|<T_0$ with
$T_0$ sufficiently small, the Schr\"{o}dinger propagator $U(t,s)$
is an OIO of the type considered in Section
\ref{sec:ganerating-functions}. The corresponding tame phase
$S(t,s,x,y)$ is
associated with the tame canonical transformation $\chi_{t,s}$,
and the amplitude $a(t,s,x,y)$ belongs to $S^0_{0,0}(\rdd)$,
depending on $t,s$.

\begin{remark}\label{rem:ja}
  These results were proved in \cite{yajima1991} in the absence
  of an electric potential, but it was observed in \cite[Remark
  (a)]{yajima1991} that they extend to Hamiltonians as in
  Definition \ref{def:yajima} by a suitable gauge transformation.
  For the benefit of the reader, we briefly provide the details
  here. If we set
  \begin{equation*}
    u'(t,x)=\exp\Big(i\int_0^t V(\tau,x)\, d\tau\Big)u(t,x),
  \end{equation*}
  then $u'$ satisfies the Schr\"{o}dinger equation with
  potentials
  \begin{equation*}
    A'(t,x):=A(t,x)+\int_0^t \nabla_x V(\tau,x)\, d\tau,
    \qquad V':=0.
  \end{equation*}
  The Hamiltonian flow $\chi_{t,s}$ remains the same after this
  transformation. This means that if
  \begin{equation}\label{eq:new-momenta}
    \xi'=\xi+\int_0^t \nabla_x V(\tau,x)\, d\tau,
    \qquad
    \eta'=\eta+\int_0^s \nabla_y V(\tau,y)\, d\tau
  \end{equation}
  denote the new momenta, we have $(x,\xi')=\chi_{t,s}(y,\eta')$.
  In view of \cite[Theorem 4]{yajima1991}, there exists a
  generating function $S'(t,s,x,y)$, which is a tame phase
  defined for $|t-s|>0$ small enough, $x,y\in\rd$, such that the
  corresponding Schr\"{o}dinger propagator $U'(t,s)$ is an OIO
  with phase $S'$. It is clear that the propagator $U(t,s)$ of
  the original equation is therefore an OIO with the tame phase
  \begin{equation*}
    S(t,s,x,y):=
    S'(t,s,x,y)-\int_0^t V(\tau,x)\, d\tau
    + \int_0^s V(\tau,y)\, d\tau.
  \end{equation*}
  Moreover, $S$ is the generating function for the map
  $(x,\xi)=\chi_{t,s}(y,\eta)$, as one can see from the equations
  \begin{equation*}
    \frac{\partial S'}{\partial y}=-\eta',
    \qquad
    \frac{\partial S'}{\partial x}=\xi',
  \end{equation*}
  and \eqref{eq:new-momenta}.
\end{remark}

By Proposition \ref{pro:CT}, for $0<|t-s|<T_0$, $U(t,s)$
therefore belongs to the class $FIO(\chi_{t,s})$, that is, for
every $N\in\mathbb{N}$ there exists $C_N$ (depending on $s,t$)
such that
\begin{equation}\label{eq:uts}
  \bigl|\langle U(t,s)\pi(z)g,\pi(w)g\rangle\bigr|
  \le C_N\,(1+|w-\chi_{t,s}(z)|)^{-N},
  \quad \forall z,w,\in\rdd,
\end{equation}
The constant $C_N$ is, in fact, uniform with respect to $t,s$ if
$|t-s|<T_0$, but we do not need this fact.

If $|t-s|$ is large, $U(t,s)$ generally cannot be written in the
form of an OIO. However, using the group property of $U(t,s)$ and
$\chi_{t,s}$, together with Proposition \ref{pro:composition}
repeatedly, we see that \eqref{eq:uts} continues to hold for
every $s,t\in\R$, with a constant $C_N$ depending on $s,t$. We
summarize this conclusion in the following proposition.

\begin{proposition}\label{pro:uts-bis}
  Let $H$ be a Yajima-type Hamiltonian. For every $s,t\in\R$, the
  corresponding propagator $U(t,s)$ belongs to the class
  $FIO(\chi_{t,s})$.
\end{proposition}

We are in the condition to apply Theorem \ref{the:compact-FIO} to
$U(t,s)$. For fixed $s,t\in\R$, the relevant hypothesis is the
vertical nonrefocusing condition for the Hamiltonian flow:
\begin{equation}\label{eq:2star-bis}
  \begin{aligned}
    &\text{There exists an open conic neighborhood }
    \Gamma\subset \rdd\setminus\{0\}
    \text{ of }\{0\}\times(\mathbb{R}^d\setminus\{0\}) \\
    &\text{and $R>0$ such that } \eta\in\rd,\ |\eta|>R
    \;\Rightarrow\; \chi_{t,s}(0,\eta)\not \in \Gamma.
  \end{aligned}
\end{equation}
Hence, the bicharacteristics that start at time $s$ at points of
the form $(0,\eta)$ with $|\eta|$ large enough, at time $t$ must
lie outside a conic open neighborhood of
$\{0\}\times(\mathbb{R}^d\setminus\{0\})$.

\begin{theorem}[Compactness for localized Schr\"{o}dinger
propagators]
\label{the:Propagator-Compactness}
  Consider a Yajima-type Hamiltonian and let $U(t,s)$ be the
  corresponding Schr\"{o}dinger propagator. With the above
  notation, let $s,t\in \R$ be such that the vertical
  nonrefocusing condition \eqref{eq:2star-bis} is satisfied.
  Then, if $E,F\subset\mathbb{R}^d$ are compact subsets, the
  operator $\mathbf{1}_E\,U(t,s)\,\mathbf{1}_F$ is compact in
  $L^2(\mathbb{R}^d)$.
\end{theorem}

\begin{proof}
The desired conclusion follows by combining Proposition
\ref{pro:uts-bis} and Theorem \ref{the:compact-FIO}.
\end{proof}


\begin{example}\label{exa:hamiltonian-examples}
Let us specialize the previous results to some interesting
cases.
\begin{itemize}
\item[(a)]
In the case of the quantum harmonic oscillator, $H$ is the Weyl
quantization of the symbol $\tfrac12(x^2+\xi^2)$, and one has
\begin{equation}\label{eq:chi-harmonic}
  \chi_{t,s}(y,\eta)
  =
  \begin{pmatrix}
    \cos(t-s) & \sin(t-s) \\
    -\sin(t-s) & \cos(t-s)
  \end{pmatrix}
  \begin{pmatrix}
    y \\ \eta
  \end{pmatrix}.
\end{equation}
The vertical nonrefocusing condition \eqref{eq:2star-bis} is
satisfied, provided that $t-s\neq k\pi$, $k\in\mathbb{Z}$; see
Remark \ref{rem:linear}.

\item[(b)] Consider a perturbed quantum harmonic oscillator, of
the type
\begin{equation*}
  H(x,\xi)=\tfrac12\xi^2+\tfrac12x^2+V(t,x).
\end{equation*}
with $V(t,x)$ satisfying the assumption in Definition
\ref{def:yajima} (b) and the additional condition
\begin{equation}
\label{eq:AddCondV}
  \lim_{|x|\to +\infty}(1+|x|)^{-1}
  \sup_{t\in\mathbb{R}}|\nabla_x V(t,x)|=0.
\end{equation}
Then, we now verify that the vertical nonrefocusing condition
\eqref{eq:2star-bis} is again satisfied, provided that $t-s\neq
k\pi$, $k\in\mathbb{Z}$.

Let $\chi_{t,s}$ and $\chi'_{t,s}$ be the Hamiltonian flows
associated with
\begin{equation*}
  H_0(x,\xi):=\frac12\bigl(\xi^2+x^2\bigr)
  \quad\text{and}\quad
  H'(t,x,\xi):=\frac12\bigl(\xi^2+x^2\bigr)+V(t,x).
\end{equation*}
Then, for $z\in\mathbb{R}^{2d}$, with $J$ as in \eqref{eq:J},
\begin{equation*}
  \frac{d}{dt}\chi_{t,s}(z)=J\nabla H_0\bigl(\chi_{t,s}(z)\bigr)
\end{equation*}
and
\begin{equation*}
  \frac{d}{dt}\chi'_{t,s}(z)=J\nabla H'(t,\chi'_{t,s}(z)),
\end{equation*}
with $\chi_{s,s}(z)=\chi'_{s,s}(z)=z$. Therefore,
\begin{equation*}
  \begin{aligned}
    \bigl|\chi_{t,s}(z)-\chi'_{t,s}(z)\bigr|
    &\le \bigl|\int_s^t
    \bigl|J\nabla H_0(\chi_{\tau,s}(z))
    -J\nabla H_0(\chi'_{\tau,s}(z))\bigr|\,d\tau \bigl|\\
    &\quad
    + \bigl|\int_s^t
    \bigl|J\nabla(H'-H_0)(\tau,\chi'_{\tau,s}(z))\bigr|\,
    d\tau\bigl|.
  \end{aligned}
\end{equation*}
Observe that the assumption on $V$ implies that, with
$w=(x,\xi)\in \R^d\times\R^d$,
\begin{align*}
  \bigl|J\nabla_w(H'-H_0)(t,w)\bigr|
  &\leq
  (1+|x|)^{-1}
  \sup_{\tau\in\mathbb{R}}|\nabla_x V(\tau,x)|
  \frac{1+|x|}{1+|w|}(1+|w|)
  \\
  &= \varphi(w)\,(1+|w|)
  \qquad \forall t\in\mathbb{R},\ w\in\mathbb{R}^{2d},
\end{align*}
for a function $\varphi>0$ such that $\varphi(w)\to 0$ as
$|w|\to+\infty$.

On the other hand, for every $T>0$ there exists a constant $C>0$
such that, for $t,s\in[-T,T]$ and $z\in\mathbb{R}^{2d}$,
\begin{equation*}
  C^{-1}(1+|z|)\le |\chi'_{t,s}(z)|\le C(1+|z|),
\end{equation*}
which implies
\begin{equation*}
  \bigl|J\nabla(H'-H_0)(t,\chi'_{t,s}(z))\bigr|
  \le C\,\tilde{\varphi}(z)\,(1+|z|)
\end{equation*}
for a new constant $C>0$, where
\begin{equation*}
  \tilde{\varphi}(z):=
  \sup_{t,s\in[-T,T]}\varphi\bigl(\chi'_{t,s}(z)\bigr)
  \longrightarrow 0
  \qquad \text{as } |z|\to+\infty.
\end{equation*}
Since $J\nabla H_0$ is a linear map, for another constant $C>0$
and for $t,s\in[-T,T]$, $z\in\mathbb{R}^{2d}$,
\begin{equation*}
  \bigl|\chi_{t,s}(z)-\chi'_{t,s}(z)\bigr|
  \le C\bigl| \int_s^t
  \bigl|\chi_{\tau,s}(z)-\chi'_{\tau,s}(z)\bigr|\,d\tau\bigr|
  + C\,\tilde{\varphi}(z)\,(1+|z|).
\end{equation*}
By Gronwall's inequality, we obtain
\begin{equation*}
  \bigl|\chi_{t,s}(z)-\chi'_{t,s}(z)\bigr|
  \le C\,\tilde{\varphi}(z)\,(1+|z|),
\end{equation*}
and the desired conclusion then follows from Proposition
\ref{pro:pertur}.

\item[(a')]
Example (a) in this section can be generalized by considering
Weyl quantizations of more general quadratic Hamiltonians
\begin{equation}
\label{eq:quadHam}
  H = 1/2 A \xi \cdot \xi + B x \cdot \xi + 1/2 Cx \cdot x,
\end{equation}
for $A,B,C \in \mathbb{R}^{d \times d}$, with $A = A^T, C = C^T$.
Indeed, one has
\begin{equation}\label{eq:chi-quadHam}
  \chi_{t,s}(y,\eta)
  =
  \begin{pmatrix}
    A_{t-s} & B_{t-s} \\
    C_{t-s} & D_{t-s}
  \end{pmatrix}
  \begin{pmatrix}
    y \\ \eta
  \end{pmatrix},
\end{equation}
where 
\begin{equation*}
    \begin{pmatrix}
    A_{t-s} & B_{t-s} \\
    C_{t-s} & D_{t-s}
  \end{pmatrix} = e^{(t-s) \mathbb{S}} \in Sp(d,\mathbb{R}), \qquad \mathbb{S} = \begin{pmatrix}
    B & A \\
    -C & -B^T
  \end{pmatrix}.
\end{equation*}
The vertical nonrefocusing condition \eqref{eq:2star-bis} is
satisfied when $\det(B_{t-s}) \neq 0$. There are two scenarios: either $\det(B_{t-s}) \equiv 0$, or the
exceptional times where $\det (B_{t-s}) = 0$ form a discrete set (see \cite[Section 4.3 and Section
6.1]{NicolaTrapasso2022}).

We show a concrete example of a Hamiltonian in
dimension $d=2$, more general than the one in
\eqref{eq:chi-harmonic}, for which the exceptional times form a discrete set. Consider the Weyl quantization the symbol
$\frac{1}{2}|\xi + (-x_2,x_1)|^2+\frac{1}{8}|x|^2$, i.e. a Hamiltonian of the form \eqref{eq:quadHam}, where
\begin{equation*}
  A = \begin{pmatrix} 1 & 0 \\ 0 & 1 \end{pmatrix},
  \quad
  B = \begin{pmatrix} 0 & -1 \\ 1 & 0 \end{pmatrix},
  \quad
  C = \begin{pmatrix} \frac{5}{4} & 0 \\ 0 & \frac{5}{4} \end{pmatrix}.
\end{equation*}
Explicit calculations show that
\begin{equation*}
  B_{t,s} =
  \left(
  \begin{array}{ c | c }
    \alpha_{t-s} & \beta_{t-s} \\
    \hline
    -\beta_{t-s} & \alpha_{t-s}
  \end{array}
  \right)
\end{equation*}
where
\begin{align*}
  \alpha_{\tau}
  &= \frac{\gamma}{8 \sqrt{5}}
  \left[ \sin\left( \frac{\tau \gamma}{2}\right)
  + \sin\left(\frac{\tau}{2 \gamma}\right) \right] \\
  &\quad
  - \frac{1}{8 \gamma \sqrt{5}}
  \left[ \sin\left( \frac{\tau \gamma}{2}\right)
  + \sin\left(\frac{\tau}{2 \gamma}\right) \right],
\end{align*}

\begin{equation*}
  \beta_{\tau} =
  \frac{1}{\sqrt{5}}
  \left[
  \cos\left(\frac{\tau \gamma}{2}\right)
  - \cos\left( \frac{\tau}{2 \gamma}\right)
  \right],
\end{equation*}
\begin{equation*}
  \gamma = \sqrt{9+4\sqrt{5}}.
\end{equation*}

Notice that $\det(B_{t-s}) = \alpha^2_{t-s} + \beta^2_{t-s} \neq
0$ if and only if $t-s \neq 2 k \pi (1 - \frac{2}{\sqrt{5}})
\gamma$, $k \in \mathbb{Z}$.

\item[(b')]
One can extend Example (a') to the quantum Hamiltonian obtained by
Weyl quantization of $H(x,\xi) + V(x,t)$, where $H(x,\xi)$ is a quadratic Hamiltonian
as introduced in Example (a'), $V(x,t)$ is a
real-valued potential satisfying again the assumption in
Definition \ref{def:yajima} (b) and the additional condition
\eqref{eq:AddCondV}. 
Indeed, as in Example (b), by
means of Proposition \ref{pro:pertur}, this new Hamiltonian
satisfies the vertical nonrefocusing condition
\eqref{eq:2star-bis} for the same non-exceptional times of Example (a'), i.e. when $t - s$ is such that
$\det(B_{t-s}) \neq 0$.

\item[(c)]
The results of this section could likewise be proved for
Hamiltonians defined via the Weyl quantization of a function
$a(t,x,\xi)$, with $a\in
C^0(\mathbb{R}\times\mathbb{R}^{2d};\mathbb{R})$, satisfying
\begin{equation*}
  |\partial_x^\alpha\partial_\xi^\beta a(t,x,\xi)|
  \le C_{\alpha,\beta}
  \qquad \text{for }|\alpha|+|\beta|\ge 2,\ x,\xi\in\rd.
\end{equation*}
Indeed, in \cite{tataru2004}, without using any short-time
integral representation as an OIO, it was shown that the
corresponding propagator $U(t,s)$ belongs to the class
$FIO(\chi_{t,s})$. However, observe that the Yajima-type
Hamiltonians in Definition \ref{def:yajima} do not fall in this
class, in general.

\item[(d)]
Similar conclusions could be proved, with the same technique, for
generalized metaplectic operators \cite{CGNR2014}.
\end{itemize}
\end{example}

The regularity assumptions on the electrostatic potential can be weakened. This is the content of the following
corollary.

\begin{corollary}
\label{cor:Corollary-PropComp}
  Consider a Hamiltonian $H = H_0 + W(x)$, where $H_0 =
  \frac{1}{2}(i \nabla +A(x))^2+V(x)$ is of Yajima--type, while $W
  \in C^1(\R^d;\R)$ with
  \begin{equation}
  \label{eq:assumptions-C1-potential}
    \nabla W\in L^\infty.
  \end{equation}
  Let $U_{W}(t,s)$ be the corresponding Schr\"{o}dinger
  propagator, and let $s,t\in \R$ be such that the vertical
  nonrefocusing condition \eqref{eq:2star-bis} is satisfied with
  respect to the free Hamiltonian $H_0$. Then, if
  $E,F\subset\mathbb{R}^d$ are compact subsets, the operator
  $\mathbf{1}_E\,U_W(t,s)\,\mathbf{1}_F$ is compact in
  $L^2(\mathbb{R}^d)$.
\end{corollary}

\begin{proof}
The Hamiltonian is time-independent, hence
$U_W(t,s)=U_{W}(t-s,0)$; we fix $s=0$. Define $W_{\epsilon} =
\rho_{\epsilon} *W$, where
$\rho_{\epsilon}(x)=\epsilon^{-d}\rho(x/\epsilon)$ is the usual
Friedrichs mollifier, the function $\rho$ being positive and
supported in $B^d_{1}(0)$. 

We notice that 
\begin{equation*}
  |W_{\epsilon}(x)- W(x)|
  \leq \int_{\mathbb{R}^d} \rho_{\epsilon}(y)
  |W(x-y) - W(x)| dy
  \lesssim_{d} \epsilon |\nabla W(x)|
  \lesssim \epsilon.
\end{equation*}
Thus $W_{\epsilon} \rightarrow W$ uniformly as $\epsilon
\rightarrow 0$. Moreover, by standard properties of convolutions
and \eqref{eq:assumptions-C1-potential}, $W_{\epsilon} \in
C^{\infty}(\R^d;\R)$, and
\begin{equation*}
  |\partial^{\alpha} W_{\epsilon}| \lesssim_{\epsilon} 1,
  \qquad \forall |\alpha| \geq 1.
\end{equation*}
We therefore see that $H_{\epsilon} = H_0 + W_{\epsilon}$ is a
Yajima--type Hamiltonian. By Duhamel's principle, we can write
  \begin{equation}
  \label{eq:Duhamel}
    \textstyle
    \mathbf{1}_{E}U_{W}(t,0) \mathbf{1}_{F}=
     \mathbf{1}_{E}U_{W_{\epsilon}}(t,0) \mathbf{1}_{F}
    -i\int_{0}^{t}\mathbf{1}_{E}U_{W_{\epsilon}}(t-\tau,0)
      (H-H_{\epsilon})U_{W}(\tau,0)\mathbf{1}_{F}d \tau.
  \end{equation}
  Notice that if $(t,0)$ satisfies the vertical nonrefocusing
  condition \eqref{eq:2star-bis} with respect to $H_0$, the same
  holds with respect to $H_{\epsilon}$, by the robustness
  argument in Proposition \ref{pro:pertur} (see also Example \ref{exa:hamiltonian-examples} (b) and (b')). Since $H -
  H_{\epsilon} = W - W_{\epsilon} \rightarrow 0$ uniformly as
  $\epsilon \rightarrow 0$, we deduce by formula
  \eqref{eq:Duhamel} that
  $\mathbf{1}_{E}U_{W}(t,0) \mathbf{1}_{F}$ is the
  limit for $\epsilon \rightarrow 0$ in the operator norm of
  $\mathbf{1}_{E}U_{W_{\epsilon}}(t,0) \mathbf{1}_{F}$, which is
  a sequence of compact operators by Theorem
  \ref{the:Propagator-Compactness}. This ends the proof.
\end{proof}

\begin{remark}
\label{rem approxPot}
One can handle a larger class of electrostatic perturbations by
a simple approximation procedure: the previous results hold
for Hamiltonians of the form
\begin{equation*}
  H=H_0+W(x),
\end{equation*}
where $H_0$ is of Yajima--type and accounts for the regular part
with at most quadratic electric growth, while $W$ lies in the
uniform closure of the class of potentials $W_0$
which are smooth, Yajima--type and satisfy the additional condition
\begin{equation*}
  \lim_{|x|\to +\infty}(1+|x|)^{-1}
  |\nabla_x W_0(x)|=0.
\end{equation*}
\end{remark}

\section{Proof of the main results}
\label{sec:main-results}

In this section, we state and prove the main results of the
paper.

Here and in the following, we will consider time-independent
Hamiltonians $H = H_0 + W(x)$, where $H_0 = \frac{1}{2}( i \nabla
+ A(x) )^2 + V(x)$ is of Yajima--type, i.e.\ as in Definition
\ref{def:yajima}, and $W(x) \in C^1(\R^d ; \R)$ satisfies
\eqref{eq:assumptions-C1-potential}, or more generally the conditions in Remark \ref{rem approxPot}. We also assume that $V + W
\geq -m^2$, for some $m \in \R$. Let $U(t,0) = e^{-i t H}$, $t
\in \R$, be the corresponding Schr\"{o}dinger propagator at time
$t$. Denote by $u(t) = e^{- i t H} u_0$, for $u_0 \in L^2(\R^d)$,
the solution of the Schr\"{o}dinger i.v.p. $i \partial_t u = H
u$, $u(0)=u_0$.

\begin{theorem}[Dynamical Amrein--Berthier inequality]
\label{the:dynAB}
  Let $0 \neq T \in \R$ be such that the vertical nonrefocusing
  condition \eqref{eq:2star-bis} is satisfied with $s=0$, $t =
  T$. Then, for any pair of compact sets
  $E,F\subset\mathbb{R}^d$, the following dynamical
  Amrein--Berthier inequality holds
  \begin{equation}
  \label{eq:dynAB}
    \|u(t)\|_{L^{2}}\le
    C \bigl( \|u(0)\|_{L^{2}(E^{c})}
    +\|u(T)\|_{L^{2}(F^{c})} \bigr),
    \qquad \forall t \in \R,
  \end{equation}
  for a constant $C = C(E,F,T,A,V,d)$.
\end{theorem}


\begin{remark}
  The aim of Theorem \ref{the:dynAB} is finding times $T \neq 0$
  for which inequality \eqref{eq:dynAB}, proved in Theorem 1.1 in
  \cite{DanconaFiorletta2026} for sublinear magnetic potentials
  and subquadratic electrostatic potentials, extends to the limit
  case in which the magnetic potential grows linearly and the
  electrostatic potential grows quadratically. Concrete scenarios
  in which Theorem \ref{the:dynAB} applies are, for instance, the
  Hamiltonians in Example \ref{exa:hamiltonian-examples} (a),
  (b),
  (a') and (b'), after dropping the time dependence of $V$.
  Notice that in all these examples, the additional assumptions
  over the electrostatic potential ensure its boundedness from
  below.
\end{remark}

\begin{proof}[Proof of Theorem \ref{the:dynAB}]
Given compact sets $E,F \subset \R^d$ and $T \neq 0$, we are in
the assumptions of Corollary \ref{cor:Corollary-PropComp} or Remark \ref{rem approxPot}, where we
take $s=0$, $t=T$, and we exchange $E$ and $F$. Therefore the
operator $\mathbf{1}_{F} e^{-i T H} \mathbf{1}_E$ is compact.
Then, by standard properties of compact operators, one easily
concludes that
\begin{equation}
\label{eq:loc-prop-op}
  S = \mathbf{1}_{E} e^{i T H} \mathbf{1}_{F} e^{-i T H},
\end{equation}
is compact. Moreover, trivially, since $S$ is the composition of
operators of norm $1$,
\begin{equation}
\label{eq:loc-Prop-ineq}
  \lVert S \rVert_{2 \rightarrow 2} \leq 1.
\end{equation}
We shall prove that inequality \eqref{eq:loc-Prop-ineq} is
strict, and Theorem \ref{the:dynAB} will follow. We proceed by
contradiction, assuming that equality holds in
\eqref{eq:loc-Prop-ineq}.

Define the $1$-eigenspace of $S$
\begin{equation}
  A_{1} = \{ g \in L^2: Sg = g\}.
\end{equation}

We can follow the very same arguments employed in the proof of
Theorem 1.1 in \cite{DanconaFiorletta2026} to prove the following
points.
\begin{enumerate}
\item
$A_{1} \neq \{ 0 \}$. We reproduce the easy proof of this claim
here for the benefit of the reader.

Pick a sequence $g_{j}\in L^{2}$ with $\|g_{j}\|_{L^{2}}=1$ such
that $\|Sg_{j}\|_{L^{2}}\to 1$; up to a subsequence, we can
assume that $g_j$ converges weakly to some $ g \in L^{2}$ with
$\|g\|_{L^{2}}\le1$. By compactness of $S$ this implies that
$Sg_{j}$ has a convergent subsequence $Sg_{j_{k}}\to Sg$. Thus we
have $1=\|Sg\|_{L^{2}}\le\|g\|_{L^{2}}\le1$, so that
$\|Sg\|_{L^{2}}=\|g\|_{L^{2}}=1$. Let
$h=e^{iTH}\mathbf{1}_{F}e^{-iTH}g$; since $Sg=\mathbf{1}_{E}h$
has norm 1 and $\|h\|_{L^{2}}\le1$, we see that $h$ must have
support contained in $E$, so that $Sg=h$ has actually norm 1. It
follows that $\mathbf{1}_{F}e^{-iTH}g$ has also norm 1 and by the
same argument this implies that $e^{-iTH}g$ has support in $F$.
Summing up, we can write $
Sg=\mathbf{1}_{E}e^{iTH}\mathbf{1}_{F}e^{-iTH}g=
e^{iTH}e^{-iTH}g=g$ and we conclude that $g$ is an eigenfunction
of $S$ corresponding to the eigenvalue 1; moreover, $g$ is
supported in $E$ and $e^{-iTH} g$ is supported in $F$. Notice
that, more generally, $g \in A_1$ if and only if $g$ is supported
in $E$ and $e^{- i T H} g$ is supported in $F$.

\item
One can replace $E$ and $F$ with a ball $B_{R}(0)$ of radius $R$
large enough and center $0$, containing both $E$ and $F$, in the
definition \eqref{eq:loc-prop-op}. Again, $A_1 \neq \{ 0 \}$,
since $A_1$ contains at least the $1$-eigenfunction $g$
constructed in the previous point.

Moreover, defining a family of cosine operators $\{ \cos(\tau
\sqrt{H + m^2}) \}_{\tau \in [0,\epsilon]}$, for some fixed
$\epsilon > 0$, one has that $\cos(\tau \sqrt{H + m^2}) A_1
\subset A_1$, thanks to the finite propagation speed property of
such operators; see Proposition 3.1 of
\cite{DanconaFiorletta2026}.

\item
Since $A_1$ is finite dimensional thanks to the compactness of
$S$, one can apply the so called Burnside Theorem on Matrix
Algebras (see Theorem 3.2 in \cite{DanconaFiorletta2026}), in
order to show that there exists $f \neq 0$ in $A_1$, thus
supported in $B_{R}(0)$, such that $\cos(\tau \sqrt{H + m^2}) f =
\lambda(\tau) f$, for some continuous $\lambda : [0 , \epsilon]
\rightarrow \R_{> 0}$.

\item
$f$ has to solve $(H + m^2) f = c f$, for some $c \in \R$. Since
it is compactly supported, it vanishes in a nonempty open set. By
the local unique-continuation property for magnetic
Schr\"{o}dinger operators, for instance in the form proved by
Kurata \cite{Kurata97} or Arrizabalaga--Zubeldia
\cite{ArrizabalagaZubeldia15}, this implies $f \equiv 0$. This is
a contradiction, since $f \neq 0$ by construction.
\end{enumerate}

Thus we have proved that $\lVert S \rVert_{2 \rightarrow 2} < 1$.

Now we can give the proof of the main inequality
\eqref{eq:dynAB}. First of all, we notice that the adjoint
$S^*$ trivially
satisfies
\begin{equation}
\label{eq:adjointIneq}
  \|S^{*}\|_{2 \to 2} < 1.
\end{equation}
The following chain of inequalities holds
\begin{align*}
  \|u(0)\|_{L^{2}}
  &=
  \|e^{-iTH}u(0)\|_{L^{2}}
\\
    &=
    \|e^{-iTH}u(0)\|_{L^{2}(F)}+
    \|e^{-iTH}u(0)\|_{L^{2}(F^c)}
\\ 
    &=
    \|e^{iTH}\mathbf{1}_{F}e^{-iTH}u(0)\|_{L^{2}}+
    \|u(T)\|_{L^{2}(F^c)}
\\
  &\le
  \|e^{iTH}\mathbf{1}_{F}e^{-iTH}\mathbf{1}_{E}u(0)\|_{L^{2}}+
  \|e^{iTH}\mathbf{1}_{F}e^{-iTH}\mathbf{1}_{E^c}u(0)\|_{L^{2}}+
  \|u(T)\|_{L^{2}(F^c)}
\\
  &\le
  \|S^*u(0)\|_{L^{2}}+
  \|u(0)\|_{L^{2}(E^c)}+  \|u(T)\|_{L^{2}(F^c)}
\\
  &\le\|S^*\|_{2 \rightarrow 2}\|u(0)\|_{L^{2}}+
  \|u(0)\|_{L^{2}(E^c)}+  \|u(T)\|_{L^{2}(F^c)}.
\end{align*}
Since $ \|u(t)\|_{L^{2}}=\|u(0)\|_{L^{2}}$ for all $t \in \R$, by
\eqref{eq:adjointIneq}, we get the main inequality
\eqref{eq:dynAB} with a constant
\begin{equation*}
  C=(1-\|S^*\|_{2 \rightarrow 2})^{-1}.
\end{equation*}

\end{proof}

\begin{remark}
\label{rem:PertOsc}
Let $H$ be a Hamiltonian satisfying \ref{def:yajima}, such that
$A = 0$ and $V(x)=\frac{1}{2} |x|^2 + W(x)$, for $W(x)$ bounded
from below satisfying the assumption in Definition
\ref{def:yajima} (b) and the additional condition
\eqref{eq:AddCondV}. By
means of Example \ref{exa:hamiltonian-examples} (b), for any time
$T \neq k \pi$, $k \in \mathbb{Z}$, we are in the assumptions of
Theorem \ref{the:dynAB}, and the dynamical Amrein--Berthier
inequality \eqref{eq:dynAB} holds for $T \neq k \pi$, $k \in
\mathbb{Z}$.

On the other hand, if we replace the additional assumption
\eqref{eq:AddCondV} on $W(x)$ by the requirement that it is
superlinear, in
the sense that
\begin{equation}
\label{eq:supLin1}
  |\partial^{\alpha} W (x) |
  \lesssim \langle x \rangle^{\lambda - |\alpha|},
  \qquad 1 < \lambda < 2
\end{equation}
for any multi-index $\alpha$, and that the Hessian matrix
$\partial^2 W(x)$ satisfies
\begin{equation}
\label{eq:supLin2}
  C_1 \langle x \rangle^{\lambda -2} \operatorname{Id}
  \leq \partial^2 W(x)
  \leq C_2 \langle x \rangle^{\lambda -2} \operatorname{Id},
  \qquad C_1,C_2 > 0,
\end{equation}
we claim that the dynamical Amrein--Berthier inequality
\eqref{eq:dynAB} holds for any $T \neq 0$, thus also when
$T = k \pi$, $k
\in \mathbb{Z} \setminus \{ 0 \}$.

Indeed, it is sufficient to show that the operator $S$ defined in
\eqref{eq:loc-prop-op} is compact for any time $T \neq 0$, and
then apply the same arguments of the proof of Theorem
\ref{the:dynAB}.

But this follows from the fact that the Schr\"{o}dinger
propagator $e^{- i T H}$ can be represented as an OIO with a
smooth kernel (see \cite{Doi2004} and \cite{Yajima1996}); hence,
the localized propagator $\mathbf{1}_{F} e^{- i T H}
\mathbf{1}_{E}$ is Hilbert--Schmidt, and $S$ is compact.
\end{remark}

\begin{remark}
\label{rem:obsControl}
We recall that if inequality \eqref{eq:dynAB} holds for any time
$T > 0$, then as a byproduct the observability inequality
\begin{equation}
\label{eq:obs-ineq}
  \lVert u_0 \rVert^2_{L^2} \lesssim_{E,T,A,V,d}
  \int^{T}_{0} \lVert e^{- i t H} u_0 \rVert_{L^2(E^c)} dt,
\end{equation}
holds for any compact set $E \subset \mathbb{R}^d$; see Theorem
1.9 in \cite{DanconaFiorletta2026}. By the observability
inequality \eqref{eq:obs-ineq} and the standard Hilbert
Uniqueness Method (see \cite{Zuazua2003}), one obtains the
following $L^{\infty}$ controllability property for $e^{-i T H}$:
given any compact set $E \subset \mathbb{R}^d$ and any pair
$u_0,u_T \in L^2(\mathbb{R}^d)$, we can always find $\nu \in
L^2(\mathbb{R}^d \times (0,T))$ such that the unique solution of the Cauchy
problem
\begin{equation}
\label{eq:schr2}
  i \partial_t u = H u + \mathbf{1}_{E^c} \nu,
  \qquad u(0) = u_0,
\end{equation}
reaches the target state $u_T$ at time $T$, i.e. $u(T)=u_{T}$.

In particular, recalling the previous remark \ref{rem:PertOsc},
this applies when $H$ is a Hamiltonian satisfying
\ref{def:yajima}, such that $A = 0$ and
$V(x)=\frac{1}{2} |x|^2 + W(x)$, for real-valued,
smooth perturbations $W$ satisfying \eqref{eq:supLin1}
and \eqref{eq:supLin2}.
\end{remark}

\bibliographystyle{abbrv}
\bibliography{dynamicalUP}

\end{document}